\documentclass[a4paper]{amsart}
\usepackage[latin1]{inputenc}
\usepackage{amssymb}
\usepackage{amsmath}
\usepackage{mathrsfs}
\usepackage{eufrak}
\usepackage{amsthm}
\usepackage{amsfonts}
\usepackage{textcomp}
\usepackage{graphicx}
\usepackage[pdftex]{color}
\usepackage{paralist}
\usepackage[shortlabels]{enumitem}
\usepackage{hyperref}
\usepackage{comment}
\usepackage[arrow, matrix, curve]{xy}

\newtheorem*{corollary*}{Corollary}
\newtheorem{theorem}{Theorem}[section]
\newtheorem*{theorem*}{Theorem}

\newtheorem{lemma}[theorem]{Lemma}
\newtheorem{proposition}[theorem]{Proposition}

\newtheorem*{claim*}{Claim}

\theoremstyle{definition}

\theoremstyle{remark}

\numberwithin{equation}{theorem}

\makeatletter
\renewcommand*\env@matrix[1][\
arraystretch]{%
  \edef\arraystretch{#1}%
  \hskip -\arraycolsep
  \let\@ifnextchar\new@ifnextchar
  \array{*\c@MaxMatrixCols c}}
\makeatother

\begin{document}

\title{On the finiteness of the Gorenstein dimension for Artin algebras}
\date{\today}



\author{Ren\'{e} Marczinzik}
\address{Institute of algebra and number theory, University of Stuttgart, Pfaffenwaldring 57, 70569 Stuttgart, Germany}
\email{marczire@mathematik.uni-stuttgart.de}

\begin{abstract}
In \cite{SSZ}, the authors proved that an Artin algebra $A$ with infinite global dimension has an indecomposable module with infinite projective and infinite injective dimension, giving a new characterisation of algebras with finite global dimension. We prove in this article that an Artin algebra $A$ that is not Gorenstein has an indecomposable $A$-module with infinite Gorenstein projective dimension and infinite Gorenstein injective dimension, which gives a new characterisation of algebras with finite Gorenstein dimension. We show that this gives a proper generalisation of the result in \cite{SSZ} for Artin algebras.
\end{abstract}

\maketitle
\section{Introduction}
In this article, we assume that all algebras are non-semisimple connected Artin algebras and all modules are right modules if nothing is stated otherwise.
All our modules will be assumed to be finitely generated expect for Gorenstein projective or Gorenstein injective modules, where we will explicitly mention in the definitions when such modules might be not finitely generated.
By definition, in an algebra of finite global dimension there are no modules that have infinite projective or infinite injective dimension.
It is an interesting question whether the converse also holds and in \cite{SSZ} this was proved, giving a nice characterisation of algebras with finite global dimension:
\begin{theorem} \label{SSZtheorem}
Let $A$ be an Artin algebra. In case $A$ has infinite global dimension, there exists an indecomposable module of infinite projective dimension and infinite injective dimension.
\end{theorem}

Recall that an algebra is a Gorenstein algebra in case the injective dimensions of the left and right regular modules are finite. In this case one can show that those two injective dimensions coincide and the common value is called the Gorenstein dimension of the algebra.
One goal of Gorenstein homological algebra is to generalise results from classicial homological algebra that involve the classical dimensions such as global, projective or injective dimension to their Gorenstein homological analogue.
The Gorenstein homological analogue of an algebra of finite global dimension is a Gorenstein algebra and the generalisation of projective dimension $pd(M)$ of a module $M$ is the Gorenstein projective dimension $Gpd(M)$ and the generalisation of injective dimension $id(M)$ is the Gorenstein injective dimension $Gid(M)$, we refer to the preliminaries for precise definitions. For algebras of finite global dimension, the Gorenstein dimension equals just the global dimension and the Gorenstein projective (injective) dimension of a module is simply the usual projective (injective) dimension of a module.
It is well known that in a Gorenstein algebra any indecomposable module has finite Gorenstein projective dimension and finite Gorenstein injective dimension (see for example \cite{Che} corollary 3.2.6.). The main result of this article is 

\begin{theorem} \label{mainresultintro}
Let $A$ be an Artin algebra. In case $A$ is not Gorenstein, there exists an indecomposable module of infinite Gorenstein projective dimension and infinite Gorenstein injective dimension.
\end{theorem}

We then use our main result to obtain a quick proof of \ref{SSZtheorem}.

\section{preliminaries}
We assume that the reader is familiar with the classical notions of representation theory and homological algebra of Artin algebras as can be found in the book \cite{ARS}. We will recall the needed results from Gorenstein homological algebra.
The standard reference for Gorenstein homological algebra in Artin algebras is \cite{Che}. An Artin algebra $A$ is called a \emph{Gorenstein algebra} in case the injective dimensions of the right and left regular modules are finite. In this case it can be shown that the two values coincide and the common value is called the \emph{Gorenstein dimension} of $A$. In case the left or right injective dimension of the regular module is infinite, the algebra is said to have infinite Gorenstein dimension and is non-Gorenstein. In general it is an open conjecture, called the Gorenstein symmetry conjecture, that the left injective dimension of the regular module always coincides with the right injective dimension of the regular module.
An acyclic complex of (not necessarily finitely generated) projective $A$-modules $P^{\bullet}$ is called \emph{totally acyclic} in case the complex $Hom_A(P^{\bullet},Q)$ is acyclic for each (not necessarily finitely generated) projective module $Q$.
A (not necessarily finitely generated) module $G$ is called a \emph{Gorenstein projective module} in case $G$ is the zeroth cocycle of a totally acyclic complex. A (not necessarily finitely generated) module $G$ is called a \emph{Gorenstein injective module} in case $D(G)$ is Gorenstein projective.
We now define the \emph{Gorenstein projective dimension} $Gpd(N)$ of a module $N$ as $Gpd(N) \leq n$ for a natural number $n \geq 0$ iff for each exact sequene of the form $0 \rightarrow K \rightarrow G^{n-1} \rightarrow ... \rightarrow G^1 \rightarrow G^0 \rightarrow N \rightarrow 0$ with $G^i$ Gorenstein projective, we have that $K$ is Gorenstein projective (this definition uses the characterisation in theorem 3.2.5 of \cite{Che} for the Gorenstein projective dimension).
Dually the \emph{Gorenstein injective dimension} $Gid(N)$ of a module $N$ is defined as the Gorenstein projective dimension of the module $D(N)$.
We note that a module of finite projective dimension, has its Gorenstein projective dimension equal to the projective dimension. Dually, a module of finite injective dimension has its Gorenstein injective dimension equal to the injective dimension.
We will need the following results in the next section:
\begin{lemma} \label{chenlemma}
Let $0 \rightarrow L \rightarrow M \rightarrow N \rightarrow 0$ be a short exact sequence.
Then the following holds:
\begin{enumerate}
\item $Gpd(N) \leq max(Gpd(M),1+Gpd(L))$.
\item $Gpd(L) \leq max(Gpd(M),Gpd(N)-1)$.
\item $Gpd(M) \leq max(Gpd(L),Gpd(N))$.

\end{enumerate}
\end{lemma}
\begin{proof}
See \cite{Che}, corollary 3.2.4.
\end{proof}

\begin{lemma} \label{bensonlemma}
Let $A$ be an Artin algebra, $N$ be an indecomposable $A$-module and $S$ a simple $A$-module.
Let $(P_i)$ be a minimal projective resolution of $N$ and $(I_i)$ a minimal injective coresolution of $N$. 
\begin{enumerate}
\item For $l \geq 0$, $Ext^{l}(N,S) \neq 0$ iff $S$ is a quotient of $P_l$. 
\item For $l \geq 0$, $Ext^{l}(S,N) \neq 0$ iff $S$ is a submodule of $I_l$.
\end{enumerate}
\end{lemma}
\begin{proof}
See \cite{Ben} corollary 2.5.4. 
\end{proof}

\begin{lemma} \label{Gorprodimlemma}
Let $A$ be an Artin algebra with an indecomposable non-projective Gorenstein projective module $G$. Then $G$ has infinite projective dimension and infinite injective dimension.
\end{lemma}
\begin{proof}
See \cite{DG}, corollary 2.9.
\end{proof}

\begin{lemma} \label{chenlemma2}
Let $A$ be an Artin algebra and $M$ an $A$-module.
Then $Gid(M) \geq \sup \{ t \geq 0 | Ext^t(D(A),M) \neq 0 \}$ and we have equality in case $Gid(M)$ is finite.
\end{lemma}
\begin{proof}
This is a special of the the dual of proposition 3.2.2. in \cite{Che}.
\end{proof}

\begin{proposition} \label{syzygyinjectives}
Let $A$ be an Artin algebra and $I$ an indecomposable injective $A$-module of infinite projective dimension. Then $\Omega^i(I)$ is never Gorenstein projective for all $i \geq 0$.
\end{proposition}
\begin{proof}
See proposition 2.2. in \cite{Mar}.
\end{proof}
\section{Main result}

Before we give a proof of our main result, we need two lemmas.
\begin{lemma} \label{gorprodimdimlemma2}
Let $N$ be a module of infinite Gorenstein projective dimension.
Then also $\Omega^p(N)$ has infinite Gorenstein projective dimension for any $p \geq 0$.

\end{lemma}
\begin{proof}
We prove this by induction on $p$, where the case $p=0$ is trivial since $\Omega^{0}(N)=N$.
Now assume that $p \geq 1$ and $\Omega^{p-1}(N)$ has infinite Gorenstein projective dimension. We show that $\Omega^p(N)$ has also infinite Gorenstein projective dimension by looking at the following short exact sequence:
$$0 \rightarrow \Omega^p(N) \rightarrow P_{p-1} \rightarrow \Omega^{p-1}(N) \rightarrow 0,$$
where $P_{p-1}$ is the projective cover of $\Omega^{p-1}(N)$. \newline
By \ref{chenlemma}, we have $Gpd(\Omega^{p-1}(N)) \leq max(Gpd(P_{p-1}),Gpd(\Omega^{p}(N)+1)$. Now $Gpd(P_{p-1})=0$ and thus $Gpd(\Omega^{p}(N)) < \infty$ would imply that also  $Gpd(\Omega^{p-1}(N)) < \infty$, which gives a contradiction. This proves that the modules $\Omega^{p}(N)$ have infinite Gorenstein projective dimension for $p \geq 0$.

\end{proof}

\begin{lemma} \label{lemmagorprodiminj}
Let $A$ be an Artin algebra.
\begin{enumerate}
\item Let $I$ be an indecomposable injective $A$-module with infinite projective dimension. Then $I$ also has infinite Gorenstein projective dimension.
\item Let $P$ be an indecomposable projective $A$-module with infinite injective dimension. Then $P$ also has infinite Gorenstein injective dimension.
\end{enumerate}
\end{lemma}
\begin{proof}
\begin{enumerate}
\item Let 
$$... P_i \rightarrow P_{i-1} \rightarrow ... \rightarrow P_1 \rightarrow P_0 \rightarrow I \rightarrow 0$$
be a minimal projective resolution of $I$. By \ref{syzygyinjectives}, $\Omega^i(I)$ is never Gorenstein projective for any $i \geq 0$ and thus by the definition of the Gorenstein projective dimension $I$ can not have finite Gorenstein projective dimension.
\item The proof is dual to the proof for (1).

\end{enumerate}

\end{proof}

\begin{theorem} \label{maintheorem}
Let $A$ be an Artin algebra that is not Gorenstein.
Then there is an indecomposable $A$-module with infinite Gorenstein projective dimension and infinite Gorenstein injective dimension. 

\end{theorem}
\begin{proof}
In this proof we will often use that for a module $R$, $Ext^i(R,T) \neq 0$ iff the projective cover of $T$ is a direct summand of $P_i$ when $(P_i)$ is a minimal projective resolution of $R$ and $T$ is a simple module, see \ref{bensonlemma}.
That $A$ is not Gorenstein means that we have that the projective dimension of $D(A)$ is infinite or that the injective dimension of $A$ is infinite.
We look at both cases. \newline
\underline{Case 1:} Assume first that the projective dimension of $D(A)$ is infinite and assume to the contrary that each indecomposable $A$-module has finite Gorenstein projective dimension or finite Gorenstein injective dimension.
By \ref{lemmagorprodiminj} $D(A)$ also has infinite Gorenstein projective dimension.
With $D(A)$ also the modules $\Omega^p(D(A))$ have infinite Gorenstein projective dimension by \ref{gorprodimdimlemma2} for $p \geq 0$. 
Define $W:= \{ T | T$ simple and $Ext^l(D(A),T) \neq 0$ for infinitely many $l \geq 1 \}$ and $w:= \sup \{ t \geq 1 | Ext^t(D(A),L) \neq 0$ for a simple module $L$ not in $W \}$ in case not every simple module is in $W$ and $w:=1$ else. Note that $W$ is non-empty since $D(A)$ has infinite projective dimension and by the definition of $W$, $w$ is a finite natural number.
By \ref{chenlemma2}, all simple modules in $W$ have infinite Gorenstein injective dimension.
Now let $r > w+1$ and note that $Ext^r(D(A), \bigoplus\limits_{T \in W}{T}) \neq 0$, since any simple module $L$ with $Ext^r(D(A),L) \neq 0$ is in $W$ because we assume $r > w+1$.
Now $Ext^r(D(A), \bigoplus\limits_{T \in W}{T}) \cong Ext^1(\Omega^{r-1}(D(A)),\bigoplus\limits_{T \in W}{T})$. Choose an indecomposable direct summand $M$ of $\Omega^{r-1}(D(A))$ with infinite Gorenstein projective dimension. Then also $Ext^1(M,\bigoplus\limits_{T \in W}{T}) \neq 0$, since $M$ has also infinite projective dimension and thus $Ext^1(M,\bigoplus\limits_{T \text{is simple}}{T}) \neq 0$ but also $Ext^1(M,\bigoplus\limits_{T \notin W}{T})= 0$ because we assume $r > w+1$ and $M$ is a direct summand of $\Omega^{r-1}(D(A))$.
Now choose a simple module $T \in W$ with $Ext^1(M,T) \neq 0$.
By construction, $M$ is indecomposable and has infinite Gorenstein projective dimension, $T$ is simple and has infinite Gorenstein injective dimension and there exists a non-split short exact sequence:
$$0 \rightarrow T \rightarrow E \rightarrow M \rightarrow 0.$$
Now consider the family of all non-split short exact sequences of the form 
$$0 \rightarrow T \rightarrow E \rightarrow M \rightarrow 0,$$
with $T$ being simple of infinite Gorenstein injective dimension and $M$ being indecomposable of infinite Gorenstein projective dimension and consider one such non-split short exact sequence in this family such that $M$ is of smallest possible length.
We now show that $E$ is indecomposable when the non-split short exact sequence is choosen such that $M$ is of smallest possible length. The proof follows closely the argument in the proof of the theorem in \cite{SSZ}.
Assume $E$ is decomposable and recall that we assume that every indecomposable module has finite Gorenstein projective dimension or finite Gorenstein injective dimension. Then there is an indecomposable direct summand $\Lambda$ of $E$ with infinite Gorenstein projective dimension (or else $E$ would have finite Gorenstein projective dimension, which would imply that also $M$ has finite Gorenstein projective dimension by \ref{chenlemma}). Since the short exact sequence does not split and $M$ is indecomposable, all coordinate maps from $T$ to direct summands of $E$ are non-zero and thus monomorphisms since $T$ is simple.
Thus there is a non-split short exact sequence of the form 
$$0 \rightarrow T \rightarrow \Lambda \rightarrow N \rightarrow 0,$$
with $\Lambda$ indecomposable with infinite Gorenstein projective dimension and such that $N$ has infinite Gorenstein projective dimension (or else the Gorenstein projective dimension of $\Lambda$ would be finite, since by assumption the Gorenstein projective dimension of $T$ is finite).
Then there exists an indecomposable direct summand $N'$ of $N$ with $Gpd(N')= \infty$. Now we take the pullback along a split monomorphism from $N'$ to $N$ of the short exact sequence 
$$0 \rightarrow T \rightarrow \Lambda \rightarrow N \rightarrow 0,$$
and obtain a non-split (since $\Lambda$ is indecomposable) short exact sequence of the form 
$$0 \rightarrow T \rightarrow F \rightarrow N' \rightarrow 0.$$
This contradicts our choice of a non-split short exact sequence with smallest possible length, since the length of $N'$ is strictly smaller than the length of $M$. Thus $E$ has to be indecomposable. $E$ has infinite Gorenstein projective dimension (or else $M$ would have finite Gorenstein projective dimension) and similarly $E$ has infinite Gorenstein injective dimension (or else $T$ would have finite Gorenstein injective dimension). This is a contradiction to our assumption that every indecomposable module has finite Gorenstein projective dimension or finite Gorenstein injective dimension. \newline

\underline{Case 2:} Assume now that the injective dimension of $A$ is infinite. By \ref{lemmagorprodiminj}, we then have $Gid(A)=\infty$. Let $B:=A^{op}$ be the opposite algebra of $A$.
$Gid(A)=\infty$ is equivalent to $Gpd(D(B))=\infty$. By case 1 of the proof, there exists an indecomposable $B$-module $G$ with infinite Gorenstein projective dimension and infinite Gorenstein injective dimension. Thus the indecomposable $A$-module $D(G)$ also has infinite Gorenstein projective dimension and infinite Gorenstein injective dimension.

\end{proof}

We can now give a quick proof of \ref{SSZtheorem}:
\begin{theorem}
Let $A$ be an Artin algebra of infinite global dimension. Then there exists a module $M$ of infinite projective dimension and infinite injective dimension.
\end{theorem}
\begin{proof}
Assume $A$ has infinite global dimension. \newline
There are two cases: \newline
\underline{Case 1:} $A$ is Gorenstein. Now a Gorenstein algebra with infinite global dimension has an indecomposable non-projective Gorenstein projective module $G$ (see for example theorem 2.3.7. of \cite{Che}). This module $G$ has infinite projective dimension and infininite injective dimension by \ref{Gorprodimlemma}. \newline
\underline{Case 2:} $A$ is not Gorenstein. In this case \ref{maintheorem} implies that there is an indecomposable module $G$ with infinite Gorenstein projective dimension and infinite Gorenstein injective dimension. But then $G$ clearly also has infinite projective dimension and infinite injective dimension.
\end{proof}

\end{document}